\documentclass[11pt]{amsart}
\usepackage{amscd, amsmath, amsthm, amssymb, xy, xypic, multirow, tikz-cd}
\usepackage{longtable}
\newtheorem{theorem}{Theorem}[section]
\newtheorem{lemma}[theorem]{Lemma}

\theoremstyle{definition}     

\newtheorem{example}[theorem]{Example}

\newtheorem{remark}[theorem]{Remark}
\numberwithin{equation}{section}

\def \hd #1 {\bfseries #1  \mdseries}
\def \italic #1 {\bfseries \it #1 \rm \mdseries}
\def \ra {\rightarrow}
\def \cen #1 { \begin{center} #1 \end{center}}

\def \mbz {\mathbb Z}

\def \mbc {\mathbb C}
\def \mbp {\mathbb P}

\def \mco  {\mathcal {O}}

\def \Q {${\mathbb {Q}}\,$}

\def \Pic {{\rm{Pic}}}

\def \rank {{\rm{rank}}}

\def \Sing {{\rm{Sing}}}

\begin{document}
\title
{Calabi--Yau threefolds with non-Gorenstein involutions}

\author{Nam-Hoon Lee}
\address{
Department of Mathematics Education, Hongik University
42-1, Sangsu-Dong, Mapo-Gu, Seoul 121-791, Korea
}
\email{nhlee@hongik.ac.kr}
\address{School of Mathematics, Korea Institute for Advanced Study, Dongdaemun-gu, Seoul 130-722, South Korea }
\email{nhlee@kias.re.kr}
\begin{abstract} The concept of non-Gorenstein involutions on Calabi--Yau threefolds is a higher dimensional generalization of non-symplectic involutions on $K3$ surfaces.
We present some elementary facts about Calabi--Yau threefolds with non-Gorenstein involutions.
We give a classification of the Calabi--Yau threefolds of Picard rank one with non-Gorenstein involutions whose fixed locus is not zero-dimensional.
\end{abstract}
\maketitle
\section{Introduction}
A \emph{Calabi--Yau manifold} is a compact K\"ahler manifold with
trivial canonical class such that the intermediate cohomology groups of
its structure sheaf are trivial  ($h^i (\mco_X) =0$ for $0 < i <
n=\dim (X)$).
If $\rho$ is an involution on $X$, then $\rho^*$ acts as  a multiplication by $+1$ or $-1$ on $H^0(X, \Omega_X^n)$ since $\dim H^0(X, \Omega_X^n) =1$.
In this note, an involution $\rho$ of a Calabi--Yau manifold will be
 said to be
\emph {Gorenstein} if $\rho^*(\omega) = \omega$  for each
$\omega \in  H^0(X, \Omega_X^n)$ and \emph{non-Gorenstein} otherwise.  The non-symplectic involutions on $K3$ surfaces, which are Calabi--Yau twofolds,  are non-Gorenstein  in this definition.
Note that Enriques surfaces, which form an important class of algebraic surfaces,  are  quotients of $K3$ surfaces by  fixed-point-free non-Gorenstein involutions.
Non-Gorenstein (so non-symplectic)  involutions on $K3$ surfaces have been   classified by V.V. Nikulin (\cite{Ni1, Ni2, Ni3, Ni4}) and they  have a remarkable property of the mirror symmetry of lattice polarized $K3$ surfaces, which enabled C.\ Voisin (\cite{Vo}) and C.\ Borcea (\cite{Bo}) to build their mirror pairs of Calabi--Yau threefolds.
They have also been used in producing many new examples of $G_2$-manifolds (\cite{CoHa, KoLe}).
Seeing that there are rich geometries and applications concerning
non-Gorenstein involutions on $K3$ surfaces, it seems natural to consider non-Gorenstein involutions on Calabi--Yau threefolds. In this note, we present some elementary findings about non-Gorenstein involutions on Calabi--Yau threefolds  with some examples.
It turns out that a non-Gorenstein involution is not fixed-point-free and has 16 fixed points if the fixed locus is zero-dimensional (Theorem \ref{thm1}).
So there are no  smooth three-dimensional generalizations of Enriques surfaces.
The next simplest  Calabi--Yau threefolds with non-Gorenstein involutions would be the ones that are of Picard rank one.
We  consider the case that the fixed locus is not zero-dimensional and the Picard rank of Calabi--Yau threefolds are one.
Almost nothing is known for the classification of general Calabi--Yau threefolds, even for the case  of Picard rank one.
However  ones with such non-Gorenstein involutions turn  out to be  closely related with \Q-Fano threefolds with some specific cyclic singularities. With this relation, we show that there are finitely many families of such Calabi--Yau threefolds (Theorem \ref{thm2}) and give a classification of Calabi--Yau threefolds of Picard number one with non-Gorenstein involutions with respect to  $H^3$, $H \cdot c_2$, $N$ and  $s$  (Tables \ref{t1-1}, \ref{t1-2}, \ref{t1-3}) with sharp bounds, where $H$ is an ample generator of the Picard group of the Calabi--Yau threefolds, $c_2$ is the second Chern class, $N$ is the number of isolated fixed points of the involutions and $s$ is the involution index.

\section{Non-Gorenstein involutions on Calabi--Yau threefolds}
Let $X$ be a Calabi--Yau threefold and $\rho$ be an involution on it.
Let $p$ be a fixed point of $\rho$. Borrowing arguments from \cite{Ca}, one can show that there exists a local holomorphic coordinate system $z_1, z_2, z_3$ around $p$ such that
\begin{enumerate}
\item $\rho(z_1, z_2, z_3) = (z_1, z_2, -z_3)$ if  $\rho$ is Gorenstein and
\item $\rho(z_1, z_2, z_3) = (-z_1, -z_2, -z_3)$ or $\rho(z_1, z_2, z_3) =(z_1, z_2, -z_3)$ if $\rho$ is non-Gorenstein.
\end{enumerate}
Hence  the fixed locus  of $\rho $ is a disjoint union of smooth curves if $\rho$ is Gorenstein. When $\rho$ is non-Gorenstein,   its fixed locus is a  union of a set $P$ of isolated points and a disjoint union $S_X$ of smooth surfaces:
$$X^\rho = P \cup S_X.$$

From now on, we assume that $\rho$ is non-Gorenstein.
Let us take a quotient $Y=X/{\langle \rho \rangle}$ of $X$ by $\rho$ and $\varphi: X \ra Y$ be the quotient map.
Then $Y$ is smooth at $\varphi(p)$ for $p \in S_X$. Let $p \in P$, then $Y$ has a singularity of type $\frac{1}{2}(1,1,1)$ at $\varphi(p)$. Since all the singularities are quotient singularities, $Y$ is \Q-factorial. 
Let $P=\{ p_1, p_2, \cdots, p_k\}$ and $q_i = \varphi (p_i)$.
Let $f:\widetilde X \ra X$ be the blow-up at $P$ with exceptional divisors $E_1, E_2, \cdots, E_k$ over $ p_1, p_2, \cdots, p_k$ and $g:\widetilde Y \ra Y$ be the blow-up at $\varphi(P)$ with the exceptional divisors  $F_1, F_2, \cdots, F_k$ over $q_1, q_2, \cdots, q_k$. Note that $F_i \simeq \mbp^2$ and $F_i|_{F_i} = \mco_{F_i}(-2)$.
 $\widetilde Y$ is now smooth.
It is easy to see that $\varphi$ can be extended to a double covering $\tilde \varphi: \widetilde X \ra \widetilde Y$, branched along
$$B=  {S_{\widetilde Y}} \cup  F_1 \cup \cdots \cup F_k,$$
where $S_Y=\varphi(S_X)$ and $S_{\widetilde Y}=g^{-1}(S_Y)$. We have the commutative diagram:

$$
\begin{tikzcd}
\arrow[d, "f"' ] \widetilde X \arrow[r, "\tilde \varphi"] \arrow[d] & \widetilde Y \arrow[d, "g" ] \\
X \arrow[r, "\varphi" ] &  Y
\end{tikzcd}
$$

\begin{lemma}\label{picn} Let $Y=X/{\langle \rho \rangle}$ be the quotient, where $X$ is a Calabi--Yau threefold and $\rho$ is a non-Gorenstein involution. Then
\cen{ $h^{i}(Y, \mco_Y)=0$ for $i>0$ and $\rank ( \Pic(Y) )\leq \rank ( \Pic(X) ).$}
\end{lemma}
\begin{proof}

Note
$$K_{\widetilde X} \sim f^*(K_X)+ 2 \sum_i E_i = 2 \sum _i E_i.$$
By the Hurwitz formula, we have
$$K_{\widetilde X} \sim {\tilde \varphi}^{*} \left( K_{\widetilde Y} \right ) +   S_{\widetilde X} +  \sum_i E_i,$$
where $S_{\widetilde X} ={\tilde \varphi}^{-1}(S_{\widetilde Y})$.
Note
$${\tilde \varphi }_* ( K_{\widetilde X} ) \sim 2 K_{\widetilde Y} + S_{\widetilde Y} + \sum_i F_i. $$
But
$$ {\tilde \varphi }_* ( K_{\widetilde X} ) \sim  {\tilde \varphi }_* \left( 2 \sum_i E_i \right) =  2 \sum_i {\tilde \varphi }_* (E_i) = 2 \sum_i F_i.$$
Hence
\begin{align}
\label{eqn2}2 K_{\widetilde Y} +  S_{\widetilde Y} \sim \sum_i F_i.
\end{align}
Let $D = K_{\widetilde Y} +S_{\widetilde Y}$, then
$$2D =  2K_{\widetilde Y} + 2S_{\widetilde Y}  = ( 2K_{\widetilde Y} + S_{\widetilde Y} ) +S_{\widetilde Y} \sim    \sum_{i} F_i+ S_{\widetilde Y},$$
which is the branch locus of the double covering $\tilde \varphi: \widetilde X \ra \widetilde Y$.
So  $${\tilde \varphi}_* (\mco_{\widetilde X} )\simeq  \mco_{\widetilde Y} \oplus \mco_{\widetilde Y}(-D).$$
Accordingly we have
\begin{align*}
H^i(\widetilde X, \mco_{\widetilde X} ) &\simeq H^i(\widetilde Y, {\tilde \varphi}_*(\mco_{\widetilde X}) ) \\
                                          &\simeq  H^i(\widetilde Y, \mco_{\widetilde Y} ) \oplus  H^i(\widetilde Y, \mco_{\widetilde Y}{(-D)} )\\
                                          &= H^i(\widetilde Y, \mco_{\widetilde Y} ) \oplus  H^i(\widetilde Y, \mco_{\widetilde Y}{(-K_{\widetilde Y} -S_{\widetilde Y})} )\\
                                           &= H^i(\widetilde Y, \mco_{\widetilde Y} ) \oplus  H^i(\widetilde Y, \mco_{\widetilde Y}{(2K_{\widetilde Y} +S_{\widetilde Y})} )^*\\
                                          &\simeq H^i (\widetilde Y, \mco_{\widetilde Y} ) \oplus  H^{3-i}\left(\widetilde Y, \mco_{\widetilde Y} {\left( \sum_i F_i \right)} \right)^*. \,\,\,\,\,\,\,\,(\because (\ref{eqn2}))
\end{align*}

Let $F=\cup_i F_i$ and consider the following exact sequence,
$$0 \ra \mco_{\widetilde Y} \ra \mco_{\widetilde Y}\left(\sum_i F_i \right) \ra \mco_{F}\left(\sum_i F_i\right) \ra 0 $$
to induce an exact sequence
$$H^{i-1}\left(F, \mco_{F}\left(\sum_i F_i\right)\right) \ra H^i \left(\widetilde Y, \mco_{\widetilde Y}\right) \ra
H^i \left(\widetilde Y,  \mco_{\widetilde Y}\left(\sum_i F_i\right) \right) \ra  H^{i}\left(F, \mco_{F}\left(\sum_i F_i\right)\right).$$
Note $ \mco_{F}\left(\sum_i F_i\right) \simeq \bigoplus_i  \mco_{F_i}(-2)$ and
$$H^{i}\left(F, \mco_{F}\left(\sum_i F_i\right)\right) \simeq \bigoplus_i H^i\left(F_i, \mco_{F_i}(-2)\right) =0$$
for any $i$, which implies
$$ H^i (\widetilde Y, \mco_{\widetilde Y}) \simeq
H^i \left(\widetilde Y,  \mco_{\widetilde Y}\left(\sum_i F_i\right) \right).$$
Therefore we have

$$H^i(\widetilde X, \mco_{\widetilde X} ) \simeq H^i (\widetilde Y, \mco_{\widetilde Y} ) \oplus  H^{3-i}\left(\widetilde Y, \mco_{\widetilde Y} {\left( \sum_i F_i \right)} \right)^*
\simeq  H^i (\widetilde Y, \mco_{\widetilde Y} ) \oplus  H^{3-i}(\widetilde Y, \mco_{\widetilde Y} )^*  .$$
Accordingly
\begin{align}
\label{eqn1}
h^i(\widetilde X, \mco_{\widetilde X} ) =  h^i (\widetilde Y, \mco_{\widetilde Y} ) +  h^{3-i} (\widetilde Y, \mco_{\widetilde Y} ).
\end{align}
We note
 $$h^p(\widetilde X, \mco_{\widetilde X}) = h^p(X, \mco_{ X}) = 0 $$
 for $p=1, 2$.
By letting $i=0$ in  (\ref{eqn1}), we have  $ h^3 (\widetilde Y, \mco_{\widetilde Y} ) =0$ and by letting $i=1$, we have
$$0=h^1(\widetilde X, \mco_{\widetilde X} ) =  h^1 (\widetilde Y, \mco_{\widetilde Y} ) +  h^2 (\widetilde Y, \mco_{\widetilde Y} ),$$
which implies $h^1 (\widetilde Y, \mco_{\widetilde Y} ) =h^2 (\widetilde Y, \mco_{\widetilde Y} ) =0$

Finally we have
$$h^i(Y, \mco_Y) = h^i (\widetilde Y, \mco_{\widetilde Y} ) =0$$
for $i>0$.

Let $\tilde \rho$ be the involution on $\widetilde X$ that is induced by the involution $\rho$ on $X$. From \cite{EsVi}, one can see

$$H^q(\widetilde X, \Omega^p_{\widetilde X}) \simeq
H^q(\widetilde Y, \tilde \varphi_* \Omega^p_{\widetilde X}) \simeq
H^q(\widetilde Y,  \Omega^p_{\widetilde Y} ) \oplus H^q(\widetilde Y, \Omega_{\widetilde Y}^p (\log B) \otimes L^{-1}),
$$
where $L$ is the anti-invariant part of the direct image of $\mco_{\widetilde X}$ to $\widetilde Y$ under the action by $\tilde \rho$.
Hence
\begin{align}
\label{neweqq} \dim H^2(\widetilde X) = \dim H^1(\widetilde X, \Omega^1_{\widetilde X}) \ge \dim H^1(\widetilde Y,  \Omega^1_{\widetilde Y} ) = \dim H^2(\widetilde Y).
\end{align}
Now we have $H^i(\mco_{\widetilde X})=0$, $H^i(\mco_{\widetilde Y})=0$ for $i=1, 2$. So by the exponential sequence, we have $\Pic(X) \simeq H^2(X, \mbz)$, $\Pic(\widetilde X) \simeq H^2(\widetilde X, \mbz)$, $\Pic(\widetilde Y) \simeq H^2(\widetilde Y, \mbz)$ and
\begin{align*}
\rank (\Pic(Y)) &= \rank ( \Pic(\widetilde Y)) - k \\
            &=  \dim H^2(\widetilde Y) - k \\
            &\leq \dim H^2(\widetilde X) - k  \,\,\,\,\,\,\,\, (\because \,\,\, \ref{neweqq}) \\
            &= \dim H^2(X)+k - k \\
            &= \rank ( \Pic (X)).
\end{align*}
i.e.\ $\rank ( \Pic(Y) ) \leq \rank ( \Pic(X)).$

\end{proof}

Let us consider some examples.

\begin{example}\label{enr} Let $X$ be the intersection of the Fermat quadric and the Fermat quartic in $\mbp^5$. Then $X$ is a Calabi--Yau threefold.
Define an involution $\rho$ by
$$(x_0, x_1, x_2, x_3, x_4, x_5) \mapsto (-x_0, -x_1, -x_2, x_3, x_4, x_5).$$
Its fixed locus is composed of 16 isolated points --- so zero-dimensional. Hence $\rho$ is non-Gorenstein (see Remark \ref{rem26}).
The quotient  $Y$ has 16 singularities of type $\frac{1}{2}(1,1,1)$.
\end{example}
\begin{example} \label{enr2}
There is a $K3$ surface $S$ with two commuting involutions $\sigma_1, \sigma_2 $ such that $\sigma_1$ is a fixed point free non-symplectic involution and $\sigma_2$ is a symplectic involution (e.g. V.\ 23 in \cite{BaHu}). Note that $\sigma_2$ has 8 fixed points.
Let $E = \mbc/(\mbz+\tau\mbz)$ be an elliptic curve with period $\tau$ and define involutions  $\theta_1, \theta_2$ of $E$ by
$$\theta_1(x) = -x, \theta_2(x) = \frac{\tau}{2}-x$$
respectively. Then the involution $\sigma_1 \times \theta_1 $ on $S \times E$ is fixed-point-free and the quotient $X:=(S\times E) / \langle \sigma_1 \times \theta_1 \rangle$ is a Calabi--Yau threefold with infinite fundamental group.
The involution $\sigma_2 \times \theta_2$ on $S \times E$ induces an  involution $\rho$ on $X$ with 16 fixed points and the quotient $Y$ has 16 singularities of type $\frac{1}{2}(1,1,1)$. Note that $\rho$ is non-Gorenstein.
\end{example}

Let us note a simple fact about topology:
\begin{lemma} \label{top}
Let $Z$ be a topological space, $q_1,q_2, \cdots, q_k$ be points on it and $Z^* = Z - \{ q_1,q_2, \cdots, q_k\}$.
If each point $q_i$ has a simply-connected neighborhood $U_i$ such that $U_i-\{q_i\}$ is open and path-connected, then the group homomorphism between fundamental groups:
$$\pi_1(Z^*) \ra \pi_1(Z),$$
which is induced by the inclusion $Z^* \hookrightarrow Z$, is surjective.
\end{lemma}
\begin{proof}
Let $Z_0 = Z^*$ and $Z_i=Z^* \cup \{q_1,q_2, \cdots, q_i \}$, then $Z_k=Z$.

We apply the Seifert--Van Kampen theorem to $Z_{i+1} = Z_i \cup U_{i+1}$, noting that $Z_i \cap  U_{i+1} = U_{i+1}-\{q_{i+1}\}$. Since  $U_{i+1}$ is simply-connected, the homomorphism
$$ \psi_{i+1}:\pi_1(Z_i) \ra \pi_1(Z_{i+1}),$$
 induced by the inclusion  $Z_i \hookrightarrow Z_{i+1}$, is surjective. Since the homomorphism $\pi_1(Z^*) \ra \pi_1(Z)$ is the composition $$\psi_{k}\circ \psi_{k-1} \circ \cdots \circ   \psi_0,$$
it is surjective.
\end{proof}

The same number 16 (the number of fixed points) in Example \ref{enr}, \ref{enr2} is not accidental:
\begin{theorem} \label{thm1}
Let $\rho$ be a non-Gorenstein involution on a Calabi--Yau threefold $X$.
Then it is not fixed-point-free. Assume that  its fixed locus is zero-dimensional. Then  the number of fixed points is 16. Furthermore, if $X$ has finite fundamental group, then the quotient $Y = X/\langle \rho \rangle$ is simply-connected.
\end{theorem}
\begin{proof}
By Lemma \ref{picn}, $\chi(Y, \mco_Y) = 1-0+0=1$. Suppose that $\rho$ is fixed-point-free, then  $2 \chi(Y,\mco_{Y} ) = \chi(X, \mco_X)=0$, which is a contradiction. So $\rho$ is not fixed-point-free.
Now assume that  its fixed locus is zero-dimensional and let $k$ be the number of the fixed points.
Note that $\chi(\widetilde Y, \mco_{\widetilde Y}) = \chi(Y, \mco_Y)=1$. By the Riemann--Roch theorem, we have
$$\chi(\widetilde Y, \mco_{\widetilde Y})  = \frac{1}{24} c_1 ( \widetilde Y ) \cdot c_2( \widetilde Y ).$$
By adjunction,
$$c( \widetilde Y )|_{F_i} = (1+F_i|_{F_i} ) \cdot c(F_i) \text{ in } \bigoplus_{k=0}^2 H^{2k}(F_i, \mbz)$$
and we have
$$ c_2 ( \widetilde Y )|_{F_i} = F_i|_{F_i} \cdot c_1(F_i) + c_2 (F_i) = -6+3=-3.$$
From  (\ref{eqn2}),
$c_1(\widetilde Y)  = -\frac{1}{2}\sum_i F_i$ and we have
\begin{align*}
\chi(\widetilde Y, \mco_{\widetilde Y})  &= \frac{1}{24} c_1 ( \widetilde Y ) \cdot c_2( \widetilde Y )\\
                    &=\frac{1}{24} \left ( -\frac{1}{2}\sum_i F_i \right ) \cdot   c_2( \widetilde Y )\\
                    &=-\frac{1}{48} \left ( \sum_i  c_2( \widetilde Y )|_{F_i} \right )  \\
                    &= -\frac{1}{48} (-3 k)\\
                    &= \frac{k}{16}.
\end{align*}

So we have $k=16$.

Now suppose that $X$ has a finite fundamental group $\pi_1(X)$.
Let
\cen{$X^*= X - \{p_1, p_2, \cdots, p_{16} \}$ and $Y^*= Y - \{q_1, q_2, \cdots, q_{16} \}$,}
where $p_i$'s are the fixed points of $\rho$ and $q_i$'s are the image of $p_i$'s in $Y$ respectively.
Applying  the Seifert--Van Kampen theorem repeatedly as in the proof of Lemma \ref{top}, one can show  $\pi_1(X^*) \simeq \pi_1(X)$ --- one can take $U_i$ so that    $U_i - \{p_i \}$ is simply-connected. Hence  $\pi_1(X^*)$ is also finite.
By Lemma \ref{top}, we know that the homomorphism $\pi_1(Y^*) \ra \pi(Y)$ is surjective. So $\pi_1(Y)$ is also finite and  $Y$ has a projective universal covering $\hat Y$.
Note that $2K_{\hat Y} =0$ and $\hat Y$ has $16 d$ singularities, where $d$ is the degree of the covering $\hat Y \ra Y$.
Let $\hat Y^* = \hat Y - \Sing(\hat Y)$, where $\Sing(\hat Y)$ is the set of the singularities of $\hat Y$.
By Theorem 1.1 in \cite{Lee1} (its proof works for the case that $-2K_Y=0$ and $S = \emptyset$), there is a smooth threefold $\hat X$ with trivial canonical class such that there is a double covering $\hat \varphi: \hat X \ra  \hat Y,$
 branched at $\Sing (\hat Y)$.
Let $\hat X^* = {\hat \varphi}^{-1}(\hat Y^*)$, then  $\hat X^* \ra \hat Y^*$ is  an unbranched covering. Thus $\pi_1(\hat X^*)$ is finite and so is $\pi_1(\hat X)$ since $\pi_1(\hat X^*) \simeq \pi_1(\hat X)$.
Hence $h^1(\hat X, \mco_{\hat X})=0$. Using
$$H^2(\hat X, \mco_{\hat X}) \simeq H^1(\hat X, \mco_{\hat X}(K_{\hat X}))^* = H^1(\hat X, \mco_{\hat X})^*,$$
 we have $h^2(\hat X, \mco_{\hat X})= h^1(\hat X, \mco_{\hat X})=0$ and we conclude that $\hat X$ is also a Calabi--Yau threefold. The covering transformation of $\hat X \ra \hat Y$ is a non-Gorenstein involution on $\hat X$ and so $\hat Y$ has 16  singularities. Therefore we have $16d=16$, i.e.\ $d=1$, which implies that $\hat Y \ra Y$ is an isomorphism and so that  $Y$ is simply-connected.

\end{proof}
If $X$ has an infinite fundamental group, then $Y$ may not be simply-connected. In  Example \ref{enr2}, let $\hat Y = (S\times E) /  \langle \sigma_2 \times \theta_2 \rangle$, then it is easy to see that there is a unbranched double covering $\hat Y \ra Y$. So $Y$ is not simply-connected.

There are fixed-point-free Gorenstein involutions on Calabi-Yau threefolds (see, for example, \cite{BiFa, BoVi}). Hence, by Theorem \ref{thm1}, we  make the following remark.
\begin{remark} \label{rem26}
An involution on a Calabi--Yau threefold is Gorenstein if and only if it is fixed-point-free or its fixed locus is one-dimensional.
\end{remark}

Now let us consider some examples where the involutions have two-dimensional fixed locus.

\begin{example}
Let $\zeta = e^{2 \pi \sqrt{-1}/3} $ and  $E_\zeta$ the elliptic curve whose period is $\zeta$.
Let $\overline X = E_\zeta^3 / \langle \zeta \rangle$ be the quotient of $E_\zeta^3$ by the scalar multiplication by $\zeta$.
Then $\overline X $ has 27 singularities of type $\frac{1}{3}(1,1,1)$. Blow up $ \overline X$  at those singularities to get a smooth threefold $X$. It is known that $X$ is a Calabi--Yau threefold $X$ (\cite {Be}).
Let $\bar 0 \in \overline X$ be the image of $(0,0,0) \in E_\zeta^3$ and $G$ be the exceptional divisor over $\bar 0$ in the blow-up $X \ra \overline X$. Then $G \simeq \mbp^2$.
The scalar multiplication on $E_\zeta^3$ by $-1$ induces a non-Gorenstein involution $\rho$ on $X$. Then $X^\rho = G \cup P$, where $P$ is a set of $31$ points.
\end{example}

\begin{example}  \label{weigh}
 Consider smooth the threefold $X$ in a weighted projective space $ \mbp (1, 1, 1, 2, 5)$ with the equation
$$x^{10} + y^{10} + z^{10} + w^5 = t^2,$$
where $x$, $y$, $z$, $w$ and $t$ are homogeneous coordinates of weights 1,
1, 1, 2 and 5 respectively. Then $X$ is a Calabi--Yau threefold. Define an involution $\rho$ by $(x,y,z,w,t)\mapsto (x,y,z,w,-t)$.
Then it is a non-Gorenstein involution and its fixed locus is:
$$X^\rho = \{(0,0,0,0,1)\} \cup S_X,$$
where $S_X=\{ (x,y,z,w, t) \in X | x^{10} + y^{10} + z^{10} + w^5 = 0 \}.$
It is not hard to see that the quotient $Y$ is isomorphic to $\mbp  (1, 1, 1, 2)$ and the quotient map is the projection:
$$\varphi: X \ra  \mbp  (1, 1, 1, 2),$$
$$ (x, y, z, w, t) \mapsto  (x, y, z, w).$$
Note that $Y=\mbp (1, 1, 1, 2)$ has a
singularity of type $\frac{1}{2} (1, 1, 1)$ at $ (0, 0, 0, 1)$.
Note that $\mbp (1, 1, 1, 2)$ is a \Q-Fano threefold.
\end{example}
We call a \Q-Factorial threefold  \Q-Fano when it has at worst terminal singularities and its anticanonical class is ample.
Classifying general Calabi--Yau threefolds with non-Gorenstein involutions seems to be out of reach for now.
In the case of Picard rank one, one can get a boundedness result  thanks to that of \Q-Fano threefolds of Picard rank one.
\begin{theorem} \label{thm2}
There are finitely many families of Calabi--Yau threefolds of Picard rank one that have  non-Gorenstein involutions with two-dimensional fixed locus.
\end{theorem}
\begin{proof}
Let $X$ be a Calabi--Yau threefold of Picard rank one that has a non-Gorenstein involution $\rho$ with two-dimensional fixed locus and $Y = X / \langle \rho \rangle$.
By Lemma \ref{picn}, $\rank(\Pic(Y)) \le 1$ and so $\rank(\Pic(Y))=1$.
We note that $2K_{\widetilde Y} = g^*(2K_Y)+\sum_i F_i.$
By (\ref{eqn2}), $g^*(-2K_Y) = S_{\widetilde Y}$. Let $H'$ be an ample divisor of $Y$.
Since $S_{\widetilde Y} \neq \emptyset$ and it is disjoint from the exceptional divisors $F_i$'s, we have
$$H'^2 \cdot (-2K_Y) =  g^*(H'^2 )\cdot g^*(-2K_Y) = g^*(H'^2 )\cdot S_{\widetilde Y} >0.$$
Since $Y$ has Picard rank one, $-2K_Y$ is an ample divisor  and so $Y$ is a \Q-Fano threefold.
Conversely a \Q-Fano threefold of Picard rank one with singularities of type only $\frac{1}{2}(1,1,1)$ has a Calabi--Yau threefold of Picard rank one as its double covering (\cite{Lee1, Lee2, Lee3}).
Since there are finitely many families of \Q-Fano threefolds of Picard rank one (\cite{Ka}),
there are finitely many families of Calabi--Yau threefolds of Picard rank one that have  non-Gorenstein involutions with two-dimensional fixed locus.

  \end{proof}

We remark that, for some   Calabi--Yau threefold $X$ of Picard rank one, there may exist  two non-Gorenstein involutions $\rho_1, \rho_2$ on $X$ with two-dimensional fixed locus such that the pairs $(X, \rho_1)$, $(X, \rho_2)$ are not deformation equivalent. Such an example is given in \cite{Lee3} (see comments before Table \ref{t1-2} also).

As shown in the proof of Theorem \ref{thm2}, the classification of Calabi--Yau threefolds of Picard rank one that have  non-Gorenstein involutions with two-dimensional fixed locus can be deduced from that of \Q-Fano threefolds of Picard rank one with  singularities of type $\frac{1}{2}(1,1,1)$.
The \Q-Fano threefolds of Picard rank one with  singularities of type $\frac{1}{2}(1,1,1)$ have been classified in several papers (\cite{Ba, CaFl, Sa1, Sa2, Ta}). Let $Y$ be a \Q-Fano threefold of Picard rank one with $N$ singularities of type $\frac{1}{2}(1,1,1)$.

With the known classifications of \Q-Fano threefolds, we give a classification of Calabi--Yau threefolds $X$'s of Picard number one with non-Gorenstein involutions $\rho$'s whose fixed locus is not zero-dimensional;  the
classification is given  with respect to numerical invariants $H^3$, $H \cdot c_2$, $e$, $N$ and involution index $s$ in Tables \ref{t1-1}, \ref{t1-2} and \ref{t1-3}, where
\begin{enumerate}
\item  $H$ is an ample generator of the Picard group of the Calabi--Yau threefold $X$,
\item $c_2$ is the second Chern class of $X$,
\item $e$ is the topological Euler characteristic of $X$,
\item $N$ is the number of isolated fixed points of $\rho$ and
\item $s$: let $S_X$ be the two-dimensional fixed locus of $\rho$, then $S_X$ is numerically equivalent to $s H$ for some positive integer $s$.
\end{enumerate}
These invariants can be easily calculated from those of \Q-Fano threefolds, whose formulas are given in \cite{Lee1, Lee2, Lee3}. For example, $H^3 = 2 |H'^3|$, where $H'$ is a generator of the class group of $Y$.

If $N=0$, then $Y$ is a smooth Fano threefold of Picard number one whose classification is classical. Their invariants are listed in Table \ref{t1-1}.

\begin{center}
\begin{longtable}{ccccl|l}
\caption{ Invariant of $X$'s when $Y$'s are smooth Fano threefolds}
\label{t1-1} \\
\hline
$N$ & $s$ & $H^3$ & $H \cdot c_2$ & $e$            & references\\
\hline
\endfirsthead
\multicolumn{4}{c}%
{\tablename\ \thetable\ -- \textit{Continued from previous page}} \\
\hline
 $N$ & $s$ & $H^3$ & $H \cdot c_2$ & $e$            & references    \\
\hline
\endhead
\hline \multicolumn{4}{r}{\textit{Continued on next page}} \\
\endfoot
\hline
\endlastfoot
0 & 2 & 4                   & 52     & -256         &\cite{IsPr}                                        \\
0 & 2 & 8                   & 56     & -176         &\cite{IsPr}                                        \\
0 & 2 & 12                  & 60     & -144         &\cite{IsPr}                                       \\
0 & 2 & 16                  & 64     & -128         &\cite{IsPr}                                        \\
0 & 2 & 20                  & 68     & -120         &\cite{IsPr}                                        \\
0 & 2 & 24                  & 72     & -116         &\cite{IsPr}                                        \\
0 & 2 & 28                  & 76     & -116         &\cite{IsPr}                                        \\
0 & 2 & 32                  & 80     & -116         &\cite{IsPr}                                        \\
0 & 2 & 36                  & 84     & -120         &\cite{IsPr}                                        \\
0 & 2 & 44                  & 92     & -128         &\cite{IsPr}                                        \\
0 & 4 & 2                   & 32     & -156         &\cite{IsPr}                                        \\
0 & 4 & 4                   & 40     & -144         &\cite{IsPr}                                        \\
0 & 4 & 6                   & 48     & -156         &\cite{IsPr}                                        \\
0 & 4 & 8                   & 56     & -176         &\cite{IsPr}                                        \\
0 & 4 & 10                  & 64     & -200         &\cite{IsPr}                                        \\
0 & 6 & 4                   & 52     & -256         &\cite{IsPr}                                        \\
0 & 8 & 2                   & 44     & -296         &\cite{IsPr}                                        \\
\end{longtable}
\end{center}

In case of $N>0$, let $-2K_Y = r_Y H_Y$, where $H_Y$ is a primitive Cartier divisor and $r_Y \in \mathbb N$. The number $\frac{r_Y}{2}$ is called the Fano index of $Y$ and denoted by $F_Y$.

If the Fano index is greater than one, there is one single example which is $\mbp(1,1,1,2)$ (\cite{ CaFl, Sa2}). We note that it is a quotient of a degree $10$ Calabi--Yau hypersurface in $\mbp(1,1,1,2,5)$ by a non-Gorenstein involution (see Example \ref{weigh}).
In this case, $N=1$, $s=10$, $H^3=1$, $H \cdot c_2 = 34$ and $e = -288$.

If the Fano index is one, such \Q-Fano threefolds are called Fano--Enriques threefolds and classified in \cite{Ba, Sa1}. There are four families of them.
Their  Calabi--Yau double coverings were studied in \cite{Lee3}. It is notable that those Calabi--Yau double coverings are not simply-connected. Their invariants are listed in Table \ref{t1-2}. Two of those \Q-Fano threefolds have the same Calabi--Yau threefold as their double coverings and that is why Table \ref{t1-2} has only three entries (not four).

\begin{center}
\begin{longtable}{ccccl|l}
\caption{Invariant of $X$'s when $Y$'s are Fano--Enriques threefolds }
\label{t1-2} \\
\hline
$N$ & $s$ & $H^3$ & $H \cdot c_2$ & $e$            & references\\
\hline
\endfirsthead
\multicolumn{4}{c}%
{\tablename\ \thetable\ -- \textit{Continued from previous page}} \\
\hline
 $N$ & $s$ & $H^3$ & $H \cdot c_2$ & $e$            & references    \\
\hline
\endhead
\hline \multicolumn{4}{r}{\textit{Continued on next page}} \\
\endfoot
\hline
\endlastfoot
8 & 2 & 4                   & 28     & -88          & \cite{Sa1, Lee3}                             \\
8 & 2 & 8                   & 32     & -64          & \cite{Sa1, Lee3}                             \\
8 & 4 & 2                   & 20     & -72          & \cite{Sa1, Lee3}                             \\
\end{longtable}
\end{center}

If the Fano index is less than one, then it is equal to $\frac{1}{2}$. In this case,  a classification was obtained in \cite{Ta}, with respect to the  invariants $N$, $s$, $H^3$ and  $H \cdot c_2$.
The Calabi--Yau double coverings of those \Q-Fano threefolds were studied in \cite{Lee1, Lee2}. We remark that  the possible Euler characteristics are not fully determined.
 Their invariants are listed in Table \ref{t1-3}.

\medskip

\begin{center}
\begin{longtable}{ccccl|l}
\caption{Invariant of $X$'s when $Y$'s are \Q-Fano threefolds of index less than one }
\label{t1-3} \\
\hline
$N$ & $s$ & $H^3$ & $H \cdot c_2$ & $e$            & references\\
\hline
\endfirsthead
\multicolumn{4}{c}%
{\tablename\ \thetable\ -- \textit{Continued from previous page}} \\
\hline
 $N$ & $s$ & $H^3$ & $H \cdot c_2$ & $e$            & references    \\
\hline
\endhead
\hline \multicolumn{4}{r}{\textit{Continued on next page}} \\
\endfoot
\hline
\endlastfoot
1 & 2 & 5                   & 50     & -200*        & \cite{Ta}                             \\
1 & 2 & 9                   & 54     & -144*        & \cite{Ta}                             \\
1 & 2 & 13                  & 58     &              & \cite{Ta}                             \\
1 & 2 & 17                  & 62     & -108*        & \cite{Ta, Lee1}                       \\
1 & 2 & 21                  & 66     & -100*, -104* & \cite{Ta, Lee1}                       \\
1 & 2 & 25                  & 70     & -100*        & \cite{Ta, Lee1}                       \\
1 & 2 & 29                  & 74     & -96*         & \cite{Ta, Lee1}                       \\
2 & 2 & 6                   & 48     & -156*        & \cite{Ta}                             \\
2 & 2 & 10                  & 52     &              & \cite{Ta}                             \\
2 & 2 & 14                  & 56     & -96*, -100*  & \cite{Ta, Lee1}                       \\
2 & 2 & 18                  & 60     & -84*, -92*   & \cite{Ta, Lee1}                       \\
2 & 2 & 22                  & 64     & -92*         & \cite{Ta, Lee1}                       \\
2 & 2 & 30                  & 72     & -95*         & \cite{Ta, Lee1}                       \\
3 & 2 & 7                   & 46     &              & \cite{Ta}                             \\
3 & 2 & 11                  & 50     &              & \cite{Ta}                             \\
3 & 2 & 15                  & 54     & -76*, -84*   & \cite{Ta, Lee1}                       \\
3 & 2 & 19                  & 58     & -76*         & \cite{Ta, Lee1}                       \\
3 & 2 & 23                  & 62     &              & \cite{Ta}                             \\
4 & 2 & 8                   & 44     &  -88*            & \cite{Ta, Lee1, Lee2}            \\
4 & 2 & 12                  & 48     & -68*         & \cite{Ta, Lee1}                       \\
4 & 2 & 16                  & 52     & -72*         & \cite{Ta, Lee1}                       \\
5 & 2 & 9                   & 42     &              & \cite{Ta}                             \\
5 & 2 & 13                  & 46     &              & \cite{Ta}                             \\
5 & 2 & 17                  & 50     &              & \cite{Ta}                             \\
6 & 2 & 10                  & 40     &              & \cite{Ta}                             \\
6 & 2 & 14                  & 44     &              & \cite{Ta}                             \\
6 & 2 & 18                  & 48     &              & \cite{Ta}                             \\
7 & 2 & 15                  & 42     &              & \cite{Ta}                             \\
7 & 2 & 19                  & 46     &              & \cite{Ta}                             \\
\end{longtable}
\end{center}

\begin{remark} (for Table \ref{t1-3}). \label{r1}
\begin{enumerate}
\item  If  $e$ is given with `*' mark, there is an example of a Calabi--Yau threefold with  those invariants and it is possible that there may be other examples with different   topological Euler characteristics (but the other invariants are the same). Hence, in this case, the Euler characteristics are not fully determined.
\item If  $e$ is missing, no examples are known yet.
\end{enumerate}
\end{remark}

We have bounds for all Calabi--Yau threefolds  of Picard number one with non-Gorenstein involutions  whose fixed locus is not zero-dimensional:
$$ 1 \le H^3 \le 44,$$
$$ 20 \le H \cdot c_2 \le 92,$$
$$ 0 \le  N  \le 8,$$
and
$$ 2 \le s \le 10.$$
We can see from Tables \ref{t1-1}, \ref{t1-2} and \ref{t1-3} that the bounds are sharp. The bound for topological Euler characteristics has not been determined yet.

The author is very thankful to the referees for making several valuable
suggestions for the initial draft of this note. This work was supported by the National Research Foundation of Korea(NRF) grant funded by the Korea government(MSIT) (No.\ 2020R1A2C1A01004503) and 2021 Hongik University Research Fund.


\end{document}